\RequirePackage{fix-cm}
\documentclass[smallextended]{svjour3}       % onecolumn (second format)
\smartqed  % flush right qed marks, e.g. at end of proof
\usepackage{graphicx}
\usepackage{amsmath}
\usepackage{amsfonts}
\usepackage{amssymb}
\usepackage{graphicx}
\newcommand{\linf}{\mathrel{\text{\raisebox{0.1ex}{\scalebox{0.8}{$\wedge$}}}}}
\newcommand{\lsup}{\mathrel{\text{\raisebox{0.1ex}{\scalebox{0.8}{$\vee$}}}}}
\newcommand{\norm}[1]{\left\lVert #1 \right\rVert}
\newcommand{\abs}[1]{\left\lvert #1 \right\rvert}
\DeclareMathOperator{\sol}{Sol}
\begin{document}

\title{unbounded $M$-weakly and unbounded $L$-weakly compact operators%\thanks{Grants or other notes
}

%\titlerunning{$u$-$M$-weakly and $u$-$L$-weakly compact operators}        % if too long for running head

\author{Zahra Niktab         \and
        Kazem Haghnejad Azar \and
        Razi Alavizadeh      \and
        Saba Sadeghi Gavgani.
}


\institute{Z. Niktab \at
           Department of Mathematics, Sarab Branch, \textbf{Islamic Azad University},  Sarab, Iran
           \and
           K. Haghnejad Azar  \at
           Department  of  Mathematics  and  Applications, Faculty of Sciences, University of Mohaghegh Ardabili, Ardabil, Iran.
           \email{haghnejad@uma.ac.ir}
           \and
           R. Alavizadeh \at
           Department  of  Mathematics  and  Applications, Faculty of Sciences, University of Mohaghegh Ardabili, Ardabil, Iran.
           \email{ralavizadeh@uma.ac.ir}
           \and
           S. Sadeghi Gavgani \at
           Department of Mathematics, Sarab Branch, \textbf{Islamic Azad University},  Sarab, Iran
}

\date{Received: date / Accepted: date}

\maketitle

\begin{abstract}
We introduce the class of unbounded $M$-weakly operators and the class of unbounded $L$-weakly compact operators. We investigate some properties for these new classification of operators and we study relation between them and $M$-weakly compact and $L$-weakly compact operators.
We also present an operator characterization of Banach lattices with order continuous norm.
\keywords{unbounded $M$-weakly compact \and unbounded $L$-weakly compact \and unbounded norm convergence \and $M$-weakly compact \and $L$-weakly compact}
\subclass{46B42 \and 47B60}
\end{abstract}
%%%%%%%%%%%%%%%%%%%%%%%%%
%%%%%%%%%%%%%%%%%%%%%%%%%
\section{Introduction}
A continuous operator $T:E\to X$ from a Banach lattice $E$ to a Banach space $X$ is said to be $M$-weakly compact if $\norm{Tx_n}\to 0$ holds for every norm bounded disjoint sequence $\{x_n\}$  of $E$.
A continuous operator $T:X\to E$ from a Banach space $X$ to a Banach lattice $E$ is said to be $L$-weakly compact if $\norm{y_n}\to 0$ holds for every sequence $\{y_n\}$  of
solid hull of $T(U)$, where $U$ is
the closed unit ball of the Banach space $X$.
The classes of $L$-weakly and $M$-weakly compact operators were introduced by
Meyer-Nieberg in \cite{meyer1974uber}.
He proved some interesting properties for these classifications of operators. For example, he proved that $L$-weakly and $M$-weakly compact operators are  weakly compact, \cite[Theorem 5.61]{aliprantis2006positive}.
%In \cite{bouras2018almost} Bouras, Lhaimer and Moussa introduced the class of almost $L$-weakly and almost $M$-weakly compact operators and studied some of their properties.
In this paper, we introduce unbounded version for these classifications of operators and we prove some of their properties.
%For example, we show they contains their counterpart.
%%%%%%%%%%%%%%%%%%%%%%%%%%%%%%%%%%%%%%%

Before we state our results, we need to fix some notations and recall some definitions. Let $E$ and $F$ be two vector lattices, let $x,y\in E$ with $x\leq y$, and let the order interval $[x, y]$ be the subset of $E$ defined by $[x, y]=\{z\in E : x\leq z \leq y\}$.
A subset of $E$ is called order bounded if it is included in an order interval.
Let $T:E\rightarrow F$ be an operator between two vector lattices $E$ and $F$.  $T$ is order bounded if it maps order bounded subsets of $E$ to order bounded subsets of $F$.
If $E$ is a normed space, then by $E^{'}$ and $E^{''}$ we will denote the topological dual and topological bidual of $E$, respectively.
The vector space $E^{\sim}$ of all order bounded linear functionals on $E$ is called the order dual of $E$. The vector space $E^{\sim\sim}=(E^{\sim})^{\sim}$ will denote the order bidual of $E$.
%
% The $b$-order bounded subsets are the sets that are order bounded in $E^{\sim\sim}$.
%
%$T$ is $b$-order bounded if it maps $b$-order bounded subsets of $E$ to $b$-order bounded subsets of $F$.
%
The algebraic adjoint of $T$ will be denoted by $T':F' \rightarrow E'$, and its order adjoint will be denoted by  $T^{\sim}:F^{\sim}\rightarrow E^{\sim}$.
%
%A vector lattice $E$ is said to be discrete if it admits a complete disjoint system of discrete elements, where we say a nonzero element $x\in E$ is discrete whenever the ideal generated by $x$ coincides with the vector subspace generated by $x$.
%
A Banach lattice is a Banach space $(E, \norm{.})$ such that $E$ is a vector lattice and its norm satisfies the following property: for each $x,y\in E$, if $\abs{x}\leq\abs{y}$, then we have $\norm{x}\leq\norm{y}$. A norm $\norm{\cdot}$ of a Banach lattice $E$ is order continuous, if for each net $(x_\alpha)_{\alpha\in\Lambda}$ such that $x_\alpha\downarrow 0$, (i.e. $(x_\alpha)$ is decreasing and $\inf \{x_\alpha : \alpha \in \Lambda\}= 0$) we have $\norm{x_\alpha}\to 0$.
A Banach lattice $E$ is said to be a Kantorovich--Banach
space ($KB$-space) whenever every increasing norm bounded
sequence of $E^{+}$ is norm-convergent.
%
%If $E$ is a Banach lattice, and $x\in E^{+}$, then the principal ideal $I_x$ generated by $x$  is
%$$I_x=\left\{ y\in E : \exists \lambda > 0 \mbox{ with } \abs{y}\leq \lambda x\right\},$$
%and thus $I_x$ under the norm $\norm{\cdot}_\infty$, defined by
%$$\norm{y}_\infty=\inf\left\{ \lambda > 0 : \abs{y}\leq \lambda x\right\},\quad y\in I_x,$$
%is an \AMS{} with the unit $x$, whose closed unit ball is the order interval $[-x, x]$. 
%%%%%%%%%%%%%%%%%%
The solid hull of a subset $A$ of vector lattice $E$ is the smallest solid set including $A$. It is easy to see that
$$\sol(A)=\left\{ x\in E : \exists y\in A \mbox{ with } \abs{x}\leq \abs{y} \right\}.$$
%%%%%%%%%%%%%%%%%%%%%%%%%%

\noindent
A net $(x_\alpha)$ in $E$ is said to be
\begin{description}
\item[unbounded order
convergent] ($uo$-convergent, for short) to $x$ if for every $u\in E^+$ the net ($|x_\alpha - x|\wedge u$)
converges to zero in order.
%%%%%%%%%%
\item[unbounded norm convergent]
 ($un$-convergent, for
short) to $x$ if $\| |x_\alpha - x| \wedge u\|\to 0$ for every $u\in E^+$.
%%%%%%%%%%%%%%%
%\item[unbounded absolute weak
%convergent]
%($uaw$-convergent, for short) to $x$
%if for every $u\in E^+$ the net ($ |x_\alpha - x| \wedge u$) converges weakly to zero.
\end{description}
For sequences in $L^p(\mu)$, where $1 \leq p <\infty$ and $\mu$ is a finite measure, it is easy to see that $uo$-convergence agrees with convergence almost everywhere, see \cite{gao2017uo}.
%??????????? [UNBOUNDED ORDER CONVERGENCE AND APPLICATION TO MARTINGALES WITHOUT PROBABILITY]
%, see, $e.g.$, [???? 5, Example2].
Under the same assumptions, $un$-convergence agrees with convergence in measure, see
\cite[Example 23]{troitsky2004measure}.
% We write $L^p$ for $L^p[0, 1]$.
%%%%%%%%%%
It is clear that if $X$ is an order continuous normed lattice then $uo$-convergence implies $un$-convergence.
%%%%%%%%%%%%%%%%%%%%%%%%%%%
%%Theorem 6.4. If X is order continuous then x
%%un 􀀀! 0 implies x
%%w􀀀
%%! 0
%%for every norm bounded net (x) in X+.
%%%%%%%%%%%%%%%%%%%%%%%%%%%%%
%%%%%%%%%%%%%%%%%%%%
%
%Let $\TEX$ be an operator between Banach lattice $E$ and Banach space \XX{}. $T$ is \OWC{} (resp., \BWC{}) if it maps order bounded (resp., $b$-order bounded) subset of $E$ to relatively weak compact subset of \XX{}.
%
%$T$ is \AMC{} (resp., \BAMC{}) if it maps order bounded (resp., $b$-order bounded) subset of $E$ to relatively compact subset of \XX{}.
%
%%By $K(E,X), AM(E,X)$ and $AM_b(E,X)$ we denote the collection of compact, \AMC{} and \BAMC{} operators, respectively. Clearly we have,
%%$$K(E,X)\subset AM_b(E,X) \subset AM(E,X).$$
For an operator $T:E\rightarrow F$ between two vector lattices we say that its modulus $\abs{T}$ exists whenever
$$\abs{T}:=T\lsup (-T)$$
%exists in $\mathcal{L}_b(E,F)$.
is a well defined operator from $E$ into $F$.
%By using \cite[Theorem 1.18]{aliprantis2006positive}, for vector lattices $E$ and $F$ whenever $F$ is Dedekind complete, each order bounded operator $\TEF{}$ satisfies the following statement:
%$$\abs{T} (x)=\sup \left\{ \abs{Ty} : \abs{y} \leq x \right\},$$
%for each $x\in E^+$.
%%%%%%%%%%%%%%%%%%%%%%%%%%
An operator $T:E\rightarrow F$ between two vector lattices is said to
be a lattice (or Riesz) homomorphism whenever $T(x\lsup y)=T(x)\lsup T(y)$
holds for all $x, y \in E$.
%%%%%%%%%%%%%%%%%%%%%%%%%%%%%%%%%
An operator $T:E\rightarrow F$ between two vector lattices is called disjointness-preserving
 if $Tx\perp Ty$ for all $x,y \in E$ satisfying $x\perp y$.
By Meyer's theorem \cite[Theorem 3.1.4]{meyer2012banach}, we know that, if an order bounded operator $T:E\rightarrow F$ between two Archimedean vector lattices preserves disjointness, then its modulus exists, and
$$\abs{T}(\abs{x})=\abs{T(\abs{x})}=\abs{Tx}$$
holds for all $x\in E$. Moreover, $\abs{T}$ is a lattice homomorphism.
%%%%%%%%%%%%%%%%%%%%%%%%%%%%%%%%%%%%%%

We refer the reader to \cite{aliprantis2006positive} and \cite{meyer2012banach} for any unexplained terms from Banach lattice theory.
%END of introduction
%%%%%%%%%%%%%%%%%%%%%%%%%%%%%%%%%%%%%%%%%%%
%%%%%%%%%%%%%%%%%%%%%%%%%%%%%
%% It is called unbounded norm convergent, and called  
%%%%%%%%%%%%%%%%%%%%%%%%%%%%%%%%%%%%
%%%%It is easy to verify that if $X$ is a lattice ideal of $L^0(\Omega,\Sigma,\mu)$, then a sequence $(f_n)$ in $X$ $uo$-converges to $f\in X$ if and only if it converges $a.e.$ to $f$.
%%%%%%%%%%%%%%%%%%%%%%%%fghfg
%For sequences in $L^p(\mu)$, where $1 \leq p <\infty$ and $\mu$ is a finite measure, it is easy to see that $uo$-convergence agrees with convergence almost everywhere, see, $e.g.$, [5, Example2]. Under the same assumptions, $un$-convergence agrees with convergence in measure, see [21, Example23]. We write $L^p$ for $L^p[0, 1]$.
%%%%%
%%%%%%%%%%%%%%%%%%
%%%%%%%%%%%%%%%%%%%%%%%
%In this paper, we denote the class of all continuous operators between two vector spaces $E$ and $F$ by $L(E,F)$.
%%
%We denote the set of all $M$-weakly compact operators from Banach lattice $E$ to Banach lattice $X$ by
%$W_{m}(E,X)$ and denote the set of all $L$-weakly compact operators from Banach space $X$ to Banach lattice $E$ by $W_{l}(X,E)$.
%%%%%%%%%%%%%%%%%%%%%%%%
%%%%%%%%%%%%%%%%%%%%%%%%%%%%%
\section{Main Results}
In this section, we introduce two new concepts as unbounded $M$-weakly compact and unbounded $L$-weakly compact  operators and we  investigate some of their properties. We establish their relationships with $M$-weakly compact and  $L$-weakly compact  operators and we study lattice properties of these new concepts.  

\begin{definition}
A continuous operator $T:E\to F$ between two Banach lattices $E$ and $F$ is said to be unbounded $M$-weakly compact (or $u$-$M$-weakly compact for short) if $Tx_n\xrightarrow{un}0$ holds for every norm bounded disjoint sequence $\{x_n\}$  of $E$.
\end{definition}
%%%%%%%%%%%%%%%%%%%%%%%
\begin{definition}
A continuous operator $T:X\to E$ from a Banach space $X$ to a Banach lattice $E$ is said to be unbounded $L$-weakly compact  (or $u$-$L$-weakly compact for short) if $y_n\xrightarrow{un}0$ holds for every disjoint sequence $\{y_n\}$  of
solid hull of $T(U)$, where $U$ is
the closed unit ball of the Banach space $X$.
\end{definition}
%%%%%%%%%%%%%%%%%%%%%%%%%%%%%%%%%
%%%%%%%%%%%%%%%%%%%%%%%%%%%%%%%%%
We do not use net in above definitions since we have the following propositions.
\begin{proposition}\label{prop:no-sigma-u-m}
A continuous operator $T:E\to F$ between two Banach lattices $E$ and $F$ is $u$-$M$-weakly compact iff 
$Tx_\alpha\xrightarrow{un}0$ holds for every norm bounded disjoint net $\{x_\alpha\}$  of $E$.
\end{proposition}
%\begin{proof}
%It is obvious that if
%$Tx_\alpha\xrightarrow{un}0$ holds for every norm 
%bounded disjoint net $\{x_\alpha\}$  of $E$ then 
%$T$ is $u$-$M$-weakly compact. For the converse, 
%we prove that if there exist some norm bounded 
%disjoint net $\{x_\alpha\}\subset E$ and some $w\in F^+$ such that 
%$\norm{\abs{Tx_\alpha}\linf w}\nrightarrow 0$
%then $T$ is not $u$-$M$-weakly compact. Let $\{x_\alpha\}\subset E$ be a norm bounded disjoint net and let $w\in F^+$ such that 
%$\norm{\abs{Tx_\alpha}\linf w}\nrightarrow 0$.
%Therefore,
%there is an $\epsilon > 0$ and an increasing sequence of indices $\{\alpha_n\}_n$ such that
%$\norm{\abs{Tx_{\alpha_n}}\linf w}\geq \epsilon$.
%That is, $T$ is not $u$-$M$-weakly compact.
%%%Since $\{x_{\alpha_n}\}_n$ is a norm bounded disjoint sequence therefore 
%\end{proof}
%%%%%%%%%%%%%%%%%%%%%%
\begin{proposition}
A continuous operator $T:X\to E$ from a Banach space $X$ to a Banach lattice $E$ is $u$-$L$-weakly compact
iff $y_\alpha\xrightarrow{un}0$ holds for every disjoint net $\{y_\alpha\}$  of
solid hull of $T(U)$, where $U$ is
the closed unit ball of the Banach space $X$.
\end{proposition}
%\begin{proof}
%It can be proved by a similar argument to that in Proposition \ref{prop:no-sigma-u-m}.
%\end{proof}
%%%%%%
\noindent
In the rest of this paper, we denote by
\begin{description}
\item[$L(X,Y)$] the class of all continuous operators between two normed vector spaces $X$ and $Y$.
\item[$MW(E,X)$] the class of all $M$-weakly compact operators from a Banach lattice $E$ to a Banach space $X$.
\item[$LW(X,E)$]  the class of all $L$-weakly compact operators from a Banach space $X$ to a Banach lattice $E$.
\item[$MW_{u}(E,F)$] the class of all $u$-$M$-weakly compact operators between two Banach lattices $E$ and $F$.
\item[$LW_{u}(X,E)$] the class of all $u$-$L$-weakly compact operators from a Banach space $X$ to a Banach lattice $E$.
\end{description}
%the class of all continuous operators between two vector spaces $E$ and $F$ by $L(E,F)$.
%
%We denote the set of all $M$-weakly compact operators from Banach lattice $E$ to Banach lattice $X$ by
%$W_{m}(E,X)$ and denote the set of all $L$-weakly compact operators from Banach space $X$ to Banach lattice $E$ by $W_{l}(X,E)$.
%%%%%%%%%%%%%%
%We denote the set of all $u$-$M$-weakly compact operators from Banach lattice $E$ to Banach lattice $F$ by
%$W_{um}(E,F)$ and denote the set of all $u$-$L$-weakly compact operators from Banach space $X$ to Banach lattice $E$ by $W_{ul}(X,E)$.
%%%%%%%%%%%%%%%%%%%%%%%%%%%%%%%
For Banach lattices $E$ and $F$, and a Banach space $X$ we have the following inclusions
$$LW(X,E)\subset LW_{u}(X,E),\quad  MW(E,F)\subset MW_{u}(E,F).$$

In the next theorem we give a condition that the reverse inclusions hold. But, in general the above inclusions are proper as shown in the following example.
%%%%%%%%%%%%%
\begin{example}
Let $T:\ell_1\to c_0$ be the inclusion operator.
The sequence $\{e_n\}$ of the
standard unit vectors is a norm bounded disjoint sequence of $\ell^1$.
We have $\| T(e_n)\|=1$ for each $n$, therefore $T$ is not $M$-weakly compact. On the other hand, let $\{a_n\}$ be a norm bounded disjoint sequence in $\ell_1$. Clearly for  each $u\in c_0^+$, $\{a_n\wedge u\}$ is a disjoint sequence in $c_0$, so it follows from order continuity of $c_0$ that  $a_n\wedge u\xrightarrow{\|\cdot\|}0$.  Therefore, $T$ is $u$-$M$-weakly compact.
%%%%%%%%
Also, we have $\{e_n\}\subset T(U)$, where $U$ is
the closed unit ball of $\ell_1$, since $\norm{e_n}=1$ for each $n$, therefore $T$ is not $L$-weakly compact. But every norm bounded disjoint sequence in solid hull of $T(U)$ is $un$-convergent to zero. Hence $T$ is $u$-$L$-weakly compact.
%%%%%%%%%%%%%%%%%%%%
\end{example}
\begin{remark}\label{th:mw=mwu}
If $F$ is a Banach lattice with strong unit, then it follows from Theorem 2.3 of \cite{kandic2017un} that $un$-topology agrees with norm topology on $F$.
That is, for a sequence $\{x_n\}\subset F$ we have
$x_n\xrightarrow{un}0$
iff
$x_n\xrightarrow{\norm{\cdot}}0$. So,
\begin{enumerate}
\item
for each Banach lattice $E$ we have,
$MW_u(E,F)=MW(E,F).$
\item
for each Banach space $X$ we have,
$LW_u(X,F)=LW(X,F).$
\end{enumerate}
%for each Banach lattice $E$ we have,
%$$MW_u(E,F)=MW(E,F)\text{ and }LW_u(E,F)=LW(E,F).$$
%%%%%%%%%%%%%%%%%%%%%%%
%\begin{proof}
%It is obvious that
%$MW(E,F)\subset MW_u(E,F)$.
%Indeed, it follows from Theorem 2.3 of \cite{kandic2017un} that $un$-topology agrees with norm topology on $F$.
%That is, for a sequence $\{x_n\}\subset F$ we have
%$x_n\xrightarrow{un}0$
%iff
%$x_n\xrightarrow{\norm{\cdot}}0$. Consequently, the proof is complete.
%\end{proof}
\end{remark}
%%%%%%%%%%%%%%%%%%%%%%%%%
%\begin{question}
%Does the inverse of Remark \ref{th:mw=mwu} hold? That is, If for each Banach lattice $E$ (Banach space $X$) we have 
%$MW_u(E,F)=MW(E,F)$ ($LW_u(X,F)=LW(X,F)$)
%then, does $F$ have a strong unit?
%\end{question}
%%%%%%%%%%%%%%%%%%%%%%
The following example shows that a compact operator need not to be $u$-$L$- or $u$-$M$-weakly compact.
\begin{example}
Let $T:\ell^1\to \ell^\infty$ be an operator defined as follows
$$T(a_n)=(\sum_{n=1}^\infty a_n,\sum_{n=1}^\infty a_n,\sum_{n=1}^\infty a_n,\cdots).$$
Clearly, $T$ is of finite rank and so is a compact operator. Now, let $\{e_n\}$ be the standard basis of $\ell^1$. We see that $\norm{Te_n\linf (1,1,1,\cdots)}=1$ for all $n$, so $T$ is not $u$-$M$-weakly compact.

On the other hand, we can easily see
$$\sol(T(U))=\{x\in\ell^\infty : \abs{x}\leq (1,1,1,\cdots)\},$$
where $U$ is the closed unit ball of $\ell^1$. Therefore,  $\{e_n\}\subset \sol(T(U))$.  Now, $\norm{e_n\linf (1,1,1,\cdots)}=1$ for all $n$, hence $T$ is not $u$-$L$-weakly compact.
\end{example}
%%%%%%%%%%%%%%%%%%%%%%%%%%% 
The notions of
$M$- and $L$-weakly compact operators
are in duality to each other, see \cite[Theorem 5.64]{aliprantis2006positive}. By the following example we show that
$u$-$M$- and $u$-$L$-weakly compact operators does not have the same duality properties.
\begin{example}
By $I_{\ell^1}, I_{\ell^\infty}$ and $I_{(\ell^\infty)'}$ we denote the identity operators of $\ell^1, \ell^\infty$ and $(\ell^\infty)'$, respectively. We know that $I'_{\ell^1}=I_{\ell^\infty}$ and $I'_{\ell^\infty}=I_{(\ell^\infty)'}$. Since $\ell^1$ and $(\ell^\infty)'$ are $AL$-spaces, so they have order continuous norm. Therefore, it follows from Theorem \ref{th:ocn} that $I_{\ell^1}$ and $I_{(\ell^\infty)'}$ are both $u$-$M$- and $u$-$L$-weakly compact operators. On the other hand, $I_{\ell^\infty}$ is neither $u$-$M$- nor  $u$-$L$-weakly compact.
\end{example}
%%%%%%%%
%%%%%%%%%%%%%%%%%%%%%%
The $u$-$L$- and $u$-$M$-weakly compact operators have a lattice approximation properties like $L$- and $M$-weakly compact operators. To prove it we need a slightly modified version of Theorem 4.36 of \cite{aliprantis2006positive} that can be proved using the argument of the same theorem. Recall that a map $f:V\to W$ from a vector space $V$ to an ordered vector space $W$ is called subadditive if for each $x,y\in V$ we have $f(x+y)\leq f(x)+f(y)$.
\begin{lemma}\label{lem:lem:epsilon}
Let $T:E\rightarrow X$ be a continuous operator from a Banach
lattice $E$ to a Banach space $X$, let $A$ be a norm bounded solid subset of $E$,
and let $f:X\to \mathbb{R}$ be a norm continuous subadditive function. If $f(Tx_n)\rightarrow 0$ holds
for each disjoint sequence $\{x_n\}$ in $A$, then for each $\epsilon>0$ there exists some
$u\in E^+$ lying in the ideal generated by A such that
$$f(T(|x|-u)^+) < \epsilon$$
holds for all $x \in A$.
\end{lemma}
% which is using the map
%$\norm{\cdot\linf u}$ instead of a seminorm
%%$un$-norm
%and can be proved easily by rewriting the proof of \cite[Theorem 4.36]{aliprantis2006positive} with some modifications, we include the proof here for the sake of completeness 
\begin{theorem}\label{th:lem:epsilon}
Let $T:E\to F$ be a continuous  operator between two Banach lattices $E$ and $F$, let $A$ be a norm bounded solid subset of $E$. If $Tx_n\xrightarrow{un}0$ holds for each disjoint sequence $\{x_n\}$ in $A$, then for each $w\in F^+$ and each $\epsilon >0$ there exists some $u\in E^+$ lying in the ideal generated by $A$ such that 
$$\| \abs{T(|x|-u)^+} \wedge w \| < \epsilon$$
holds for all $x\in A$.
\end{theorem}
\begin{proof}
%Assume that $T:E\to F$ is a continuous  operator between two Banach lattices $E$ and $F$ and $A$ is a norm bounded solid subset of $E$ such that $Tx_n\xrightarrow{un}0$ holds for each disjoint sequence $\{x_n\}$ in $A$.
Let $w\in F^+$ and $\epsilon >0$ be arbitrary but fixed.
%%%%%%%%%%%%%%%%%%%%%%%%%%%
%%Let $w\in F$ be arbitrary but fixed.
Define map $f_w:F\to \mathbb{R}$ as
$$f_w(x)=\norm{\abs{x}\linf w},$$
for each $x\in F$. It is obvious that $f_w$ is subadditive and norm continuous.
%Let $x,y\in F$ be arbitrary. We have,
%\begin{eqnarray*}
%f_w(x+y)&=&\norm{\abs{x + y}\linf w}\\
%&\leq &\norm{(\abs{x} + \abs{y})\linf w}\\
%&\leq & \norm{(\abs{x}\linf w) + (\abs{y}\linf w)}\\
%&\leq & \norm{\abs{x}\linf w} + \norm{\abs{y}\linf w}\\
%&=& f_w(x)+f_w(y).
%\end{eqnarray*}
%%On the other hand, we have
%%$$\abs{f_w(x)}=\norm{\abs{x}\linf w}\leq \norm{x}.$$
%%Therefore, $f_w$ is norm continuous.
%Now, we show that $f_w$ is norm continuous. Let $x$ and $y$ be arbitrary elements in $F$. Without loss of generality, we assume that $f_w(x)\leq f_w(y)$.  It follows from subadditivity of $f_w$ that 
%$f_w(y)-f_w(x)\leq f_w(y-x)$. Therefore,
%$$\abs{f_w(y)-f_w(x)}=f_w(y)-f_w(x)\leq f_w(y-x) =\norm{\abs{y-x}\linf w}\leq \norm{y-x},$$
%hence $f_w$ is norm continuous.
%%$\abs{f_w(y)-f_w(x)}=

Let $\{x_n\}$ be a disjoint sequence in $A$. It follows from assumption and 
$f_w(Tx_n)=\norm{\abs{Tx_n}\linf w}$
that $f_w(Tx_n)\to 0$.
Therefore, by Lemma \ref{lem:lem:epsilon} there exists some $u\in E^+$ lying in the ideal generated by $A$ such that 
$$f_w(T(|x|-u)^+) < \epsilon$$
for all $x \in A$.  That is,
$$\norm{\abs{T(|x|-u)^+}\linf w} < \epsilon$$
for all $x \in A$. Thus, the proof is complete.
%By way of contradiction assume that there exists some $w\in F^+$ and some $\epsilon > 0$ such
%that for each $u \geq 0$ in the ideal generated by $A$ we have
%$\| T(|x|-u)^+ \wedge w \| \geq \epsilon$ for at least one $x\in A$.
%%%%%%%%%%%%%%%%%%%%%
%Therefore, there exists a sequence $\{x_n\}\subset A$ such that for each $n$ we have
%\begin{equation}\label{eq:1}
%\| T[(|x_{n+1}|-4^n\sum_{i=1}^{n}|x_i|)^+] \wedge w \| \geq \epsilon.
%\end{equation}
%
%Now put $y=\sum_{n=1}^\infty 2^{-n}|x_n|$. Also, let $w_n=(|x_{n+1}|-4^n\sum_{i=1}^{n}|x_i|)^+$ and $v_n=(|x_{n+1}|-4^n\sum_{i=1}^{n}|x_i|-2^{-n}y)^+$. By Lemma [4.35, positive], the sequence $\{v_n\}$ is disjoint. Also, since $A$ is solid and $0\leq v_n \leq |x_{n+1}|$ holds, we see that $\{v_n\}\subset A$, and so by our hypothesis
%$Tv_n\xrightarrow{un}0$ .
%
%On the other hand, we have $0\leq w_n-v_n \leq 2^{-n}y$, and so $\| w_n - v_n \| \leq 2^{-n}\| y \|$. In particular, it follows that $\| T(w_n-v_n) \| \xrightarrow{un}0$.
%From, $\| Tw_n \wedge w \| \leq \|T(w_n - v_n) \wedge w \| + \| Tv_n \wedge w \|,$
%we see that $Tw_n\xrightarrow{un}0$. However, this contradicts \eqref{eq:1}, and the proof is complete.
\end{proof}
%%%%%%%%%%%%%%%%%
\begin{corollary}\label{th:epsilon}
For Banach lattices $E$ and $F$, and a Banach space $X$ the following statements hold:
\begin{enumerate}
\item%
If $T:E\to F$ is a $u$-$M$-weakly compact operator, then for each $w\in F^+$ and for each $\epsilon  >0 $ there exists some $u\in E^+$ such that
$$\| \abs{T(|x|-u)^+}\wedge w \| <\epsilon$$
holds for all $x\in E$ with $\|x\|\leq 1$.
%%%%%%%%%%%%%%%%%%%%%%
\item%
If $T:X\to E$ is a $u$-$L$-weakly compact operator, then for each $w\in E^+$ and for each $\epsilon > 0$
there exists some $u\in E^+$ lying in the ideal generated by $T(X)$
satisfying
$$\| [(|Tx|-u)^+]\wedge w \| <\epsilon$$
for all $x\in X$ with $\| x\| \leq 1$.\label{th:epsilon:lw}
\end{enumerate}
\end{corollary}
%\begin{proof}
%\begin{enumerate}
%\item%
%Let $A$ be the closed unit ball of $E$. Then, for each disjoint sequence $\{x_n\}$ in $A$ we have $Tx_n\xrightarrow{un}0$. So, by Theorem \ref{th:lem:epsilon}, there exists some $u\in E^+$ satisfying
%$\| \abs{T(|x|-u)^+} \wedge w \| < \epsilon$ for all $x\in A$.
%\item%
%Let $U$ denote the closed unit ball of $X$, and let $A$ be the solid hull
%of $T(U)$. Since $T$ is a $u$-$L$-weakly compact operator, every disjoint sequence in $A$ is $un$-convergent to zero.
%Now if $I:E\to E$ is the identity operator, then it
%follows from Theorem \ref{th:lem:epsilon} that there exists some $u\in E^+$ lying in the ideal
%generated by $A$ such that
%$\| \abs{I[(|y|-u)^+]}\wedge w \| <\epsilon$
% holds for all $y\in A$. In particular,
%we have
%$\| (|Tx|-u)^+\wedge w \| <\epsilon$
%for all $x\in U$.\qedhere
%\end{enumerate}
%\end{proof}
%%%%%%%%%%%%%%%%%%%%%%%%%%%%%%%%
In the following theorem we show that the class of $u$-$M$- and $u$-$L$-weakly compact operators are closed subspaces of the vector space of continuous operators.% and satisfy the domination problem.
\begin{theorem}
For Banach lattice $E$ and $F$, and a Banach space $X$ the following hold.
$LW_u(X,E)$ and $MW_u(E,F)$  are closed vector subspaces of $L(X,E)$ and $L(E, F)$, respectively. 
%\item%
%If
%$T\in MW_u(E,F)$ and $S\in L(E, F)$ such that $0\leq S \leq T$,  then $S\in MW_u(E,F)$.
%\item%
%If
%$T\in LW_u(E,F)$ and $S\in L(E, F)$ such that $0\leq S \leq T$,  then $S\in LW_u(E,F)$.
\end{theorem}
\begin{proof}
At first we show that $LW_u(X,E)$ is a subspace of $L(X,E)$.
Obviously, $LW_u(X,E)$ is closed under the scalar multiplication. Let $T,S\in LW_u(X,E)$ we prove that $T+S\in LW_u(X,E)$.
%%%%%%%%%%%%%%%%%%%%%%%%%%%%%%%%
Let $U$ be the closed unit ball of $X$. We claim that $\sol((T+S)(U))^+\subseteq \sol(T(U))^++\sol(S(U))^+$. Indeed, we have
\begin{eqnarray*}
&&\sol((T+S)(U))^+=\\
&& \quad \{y\in E^+ : y\leq\abs{(T+S)(u)}\text{ for some }u\in U\}\\
&&=\{y_1+y_2 : y_1+y_2\leq \abs{Tu+Su}\text{ for some }u\in U\\
& &\qquad\text{ such that }0\leq y_1\leq\abs{Tu}, 0\leq y_2\leq\abs{Su}\}\\
&&\subseteq \{y_1+y_2 : y_1\in\sol(T(U))^+, y_2\in\sol(S(U))^+\}\\
&&=  \sol(T(U))^++\sol(S(U))^+
\end{eqnarray*}
%%%%%%%%%%%%%%%%%%%%%%%%%%%%%%%%%%%%%
Now, assume that $\{y_n\}\subset \sol((T+S)(U))$ be a disjoint sequence. Clearly, $\{\abs{y_n}\}$ is a disjoint sequence in 
$\sol((T+S)(U))$. By the above argument there exist $\{y_{1,n}\}\subset \sol(T(U))^+$ and $\{y_{2,n}\}\subset \sol(S(U))^+$ such that
$\abs{y_n}=y_{1,n}+y_{2,n}$ for each $n$. If $n\neq m$ then
$y_{i,n}\linf y_{i,m}\leq (y_{1,n}+y_{2,n})\linf (y_{1,m}+y_{2,m})=\abs{y_n}\linf \abs{y_m}=0$ for $i=1,2$. Thus, $\{y_{i,n}\}$ is a disjoint sequence for $i=1,2$. Since $T$ and $S$ are $u$-$L$-weakly compact we have $y_{i,n}\xrightarrow{un}0$ for $i=1,2$.
Now, for $w\in E^+$ we have
\begin{eqnarray*}
\norm{\abs{y_n}\linf w}&=&\norm{(y_{1,n}+y_{2,n})\linf w}\\
&\leq &\norm{(y_{1,n}\linf w) + (y_{2,n}\linf w)}\\
&\leq &\norm{y_{1,n}\linf w} + \norm{y_{2,n}\linf w},
\end{eqnarray*}
therefore, $y_n\xrightarrow{un}0$. That is, $T+S\in LW_u(X,E)$.
%----------------------------- To see that $W_{ul}(E,F)$ is closed, let $T\in \bar{W_{ul}(E,F)}$. Therefore, for $\epsilon > 0$ there exists $u$-$L$-weakly compact operator $S:E\to F$ such that $\|T-S\|<\epsilon$. ??????????????????????????
%\end{enumerate}
%%%%%%%%%%%%%%%%%%%%%
%%%%%%%%%%%%%%%%%%%%%%%
%Clearly, $LM_u(E,F)$ is closed in $L(E,F)$.

Now, to prove that $LW_u(X,E)$ is closed in $L(X,E)$ we show that 
$T\in \overline{LW_u(X,E)}$ is $u$-$L$-weakly compact.
Let $\epsilon >0$ and $w\in E^+$ be arbitrary but fixed from now on and let $S\in LW_u(X,E)$ such that $\norm{T-S}<\frac{\epsilon}{3}$. Assume that $\{y_n\}\subset \sol(T(U))$ be a disjoint sequence we want to show that $y_n\xrightarrow{un}0$.
For each $n\in\mathbb{N}$ choose some $u_n\in U$ such that $\abs{y_n}\leq \abs{T(u_n)}$. By using Part \ref{th:epsilon:lw} of Corollary \ref{th:epsilon} choose $u\in E^+$ lying in the ideal generated by $S(X)$ such that 
$\| (|Sx|-u)^+\wedge w \| <\frac{\epsilon}{3}$ for all $x\in U$.
We have
\begin{eqnarray*}
0 \leq \abs{y_n} &\leq & \abs{(S-T)u_n}+\abs{Su_n}\\
& \leq & \abs{(S-T)u_n}+(\abs{Su_n}-u)^++\abs{Su_n}.
% + \abs{T(u_n^{-}-u)^+}\linf w
%+ 2\abs{T}u\linf w
% (\abs{Tu_n^+}+\abs{Tu_n^-})\linf w\\
%& \leq & \abs{T(u_n^{+}-u)^+}\linf w + \abs{T(u_n^{-}-u)^+}\linf w
%+ 2\abs{T}u\linf w
\end{eqnarray*}
It follows from \cite[Theorem 1.13]{aliprantis2006positive} that for each $n\in\mathbb{N}$ there exist $v_n,w_n,z_n\in E^+$ with
$v_n\leq \abs{(S-T)u_n}, w_n\leq (\abs{Su_n}-u)^+$ and $z_n\leq \abs{Su_n}$ such that $\abs{y_n}=v_n+w_n+z_n$. Clearly, $\{z_n\}$ is a disjoint sequence in $\sol(S(U))$.
Since $S$ is $u$-$L$-weakly compact, $z_n\xrightarrow{un}0$.
Specifically, there exists $N\in\mathbb{N}$ such that we have $\norm{z_n\linf w}< \frac{\epsilon}{3}$  for all $n\geq N$.
On the other hand,
$\norm{v_n}\leq \norm{(S-T)u_n} \leq \norm{S-T}<\frac{\epsilon}{3},$
and since $\{u_n\}\subset U$, we have
$\norm{w_n\linf w}\leq \norm{(\abs{Su_n}-u)^+\linf w}<\frac{\epsilon}{3},$
for all $n\in\mathbb{N}$.
Therefore,
\begin{eqnarray*}
0 \leq \norm{\abs{y_n}\linf w} &\leq & \norm{(v_n + w_n + z_n)\linf w}\\
& \leq & \norm{v_n}+\norm{w_n\linf w}+\norm{z_n\linf w}\\
& < & \frac{\epsilon}{3} + \frac{\epsilon}{3} + \frac{\epsilon}{3} = \epsilon.
\end{eqnarray*}
As $w\in E^+$ is arbitrary hence $y_n\xrightarrow{un}0$. Therefore, $T\in LW_u(X,E)$.\\
Now, we prove that  $MW_u(E,F)$ is a norm closed vector subspace of $L(E,F)$. 
It is obvious that the set of all $u$-$M$-weakly compact operators between $E$ and $F$ is a vector subspace of $L(E,F)$. Let $T$ be in the closure of the set of all $u$-$M$-weakly compact operators, and let $\{x_n\}$ be a disjoint sequence of $E$ satisfying $\|x_n\|\leq 1$ for all $n$. We have to show 
that $Tx_n\xrightarrow{un}0$. Let $w\in F^+$ and $\epsilon > 0$ be arbitrary but fixed from now on. There exists some $u$-$M$-weakly compact operator $S:E\to F$ such that $\|T-S\|<\epsilon$. 
%Now it follows from 
We have
$|Tx_n|\leq |(T-S)x_n|+|Sx_n|$. Therefore,
\begin{eqnarray*}
|Tx_n|\wedge w & \leq & (|(T-S)x_n|+|Sx_n|)\wedge w\\
& \leq & |(T-S)x_n| + |Sx_n| \wedge w.
\end{eqnarray*}
Hence,
\begin{eqnarray*}
\| |Tx_n|\wedge w \| & \leq & \| |(T-S)x_n| + |Sx_n| \wedge w \|\\
& \leq & \| (T-S)x_n \| + \| |Sx_n| \wedge w\|.
\end{eqnarray*}

So it follows that,
$$\| |Tx_n|\wedge w \| \leq \epsilon + \| |Sx_n| \wedge w\|.$$
Since $\epsilon > 0$ is arbitrary we see that $\| |Tx_n|\wedge w \|\to 0$ holds. Now the proof follows from the fact that $w\in F^+$ is arbitrary.
\end{proof}
%%%%%%%%%%%%%%%%%%%%%%%%%
\begin{theorem}
Assume that a Banach lattice $E$ has order continuous norm and $G$ is order dense sublattice of  $E$. If $0\leq T\in MW_u(G,F)$, then $T\in MW_u(E,F)$.
\end{theorem}
\begin{proof}
Let $\{x_n\}$ be a norm bounded disjoint sequence in $E$ and $0\leq T\in MW_u(G,F)$. It follows that $\{x^-_n\}$ and $\{x^+_n\}$ are norm bounded disjoint sequence in $E$, and so  
\begin{eqnarray*}
\| |Tx_n|\wedge w \|  \leq \| T|x_n|\wedge w \|\leq \| Tx^-_n\wedge w \|+\| Tx^+_n\wedge w \|,
\end{eqnarray*}
whenever $w\in F^+$.
%By  notice the preceding inequalities,
Therefore, we may assume without loss of generality that $x_n >0$ for all $n\in \mathbb{N}$. Since $G$ is order dense in $E$, for each   $n\in \mathbb{N}$, there is a sequence $\{x_{n,k}\}$ such that $0<x_{n,k}\uparrow x_n$. It follows that $x_n-x_{n,k}\downarrow 0$, and so by assumption we have $\| x_n-x_{n,k}\| \rightarrow 0$. By continuity of $T$, we have $\| T(x_n-x_{n,k})\| \rightarrow 0$, and so $\| T(x_n-x_{n,k})\wedge w\| \rightarrow 0$ whenever $w\in F^+$.
Let $w\in F^+$, $\epsilon>0$ and $n\in \mathbb{N}$. Then there is some $k_n$ such that   $\| T(x_n-x_{n, k_n})\wedge w\|<\frac{\epsilon}{2} $.  
On the other hand, for each $k\in \mathbb{N}$ the sequence 
${\{x_{n,k_n}}\}_{n=1}^{+\infty}$ is a norm bounded disjoint sequence in $G$,  by assumption we have $\| T(x_{n,k_n})\wedge w\| \rightarrow 0$, and so there is some $N\in\mathbb{N}$ such that $\| T(x_{n,k_n})\wedge w\| <\frac{\epsilon}{2}$ for all $n\geq N$.
Thus we have 
\begin{eqnarray*}
\| Tx_n\wedge w \| \leq \|  T(x_n-x_{n,k_n})\wedge w \|+\| T(x_{n,k_n})\wedge w \|<\frac{\epsilon}{2}+\frac{\epsilon}{2}=\epsilon,
\end{eqnarray*}
for all $n\geq N$, and so the proof follows.
\end{proof}

The following theorem shows that $u$-$L$- and $u$-$M$-weakly compact operators satisfy the domination problem.
\begin{theorem}
$u$-$L$- and $u$-$M$-weakly compact operators satisfy the domination problem.
 That is,
for Banach lattice $E$ and $F$, and a Banach space $X$,
% the following hold.
if
$T\in MW_u(E,F)$ (resp. $T\in LW_u(X,E)$) and $S\in L(E, F)$ (resp. $S\in L(X,E)$)  such that $0\leq S \leq T$,  then $S\in MW_u(E,F)$ (resp. $S\in LW_u(X,E)$).
\end{theorem}
\begin{proposition}
The following assertions hold:
\begin{enumerate}
\item
If $T:E\to F$ between two Banach lattices is a lattice homomorphism that its range is norm dense in $F$ and $E$ has order continuous norm, then $T$ is $u$-$M$-weakly compact.
\item
If $T:X\to E$ from a Banach space $X$ to a Banach lattice $E$ is continuous and $E$ has order continuous norm then $T$ is $u$-$L$-weakly compact.
\end{enumerate}
\end{proposition}
\begin{proof}
\begin{enumerate}
\item
Let $\{x_n\}$ be a norm bounded disjoint sequence in $E$ and $w\in F^+$. For $\epsilon >0$ there exists $x\in E$ such that $\norm{Tx - w}< \frac{\epsilon}{2}$. On the other hand, since $E$ has order continuous norm therefore $\{x_n\}$ is $un$-null. Hence there exists $0<N\in\mathbb{N}$ such that for each $n\geq N$ we have
$\norm{\abs{x_n}\linf \abs{x}}<\frac{\epsilon}{2\norm{T}}$.
Now, for each $n\geq N$ we have
\begin{eqnarray*}
\norm{\abs{Tx_n}\linf w} & = & \norm{\abs{Tx_n}\linf \abs{w-Tx+Tx}}\\
& \leq & \norm{\abs{Tx_n}\linf \abs{w-Tx}+\abs{Tx_n}\linf \abs{Tx}}\\
& \leq & \norm{\abs{Tx_n}\linf \abs{w-Tx}}+\norm{\abs{Tx_n}\linf \abs{Tx}}\\
& = & \norm{\abs{Tx_n}\linf \abs{w-Tx}}+\norm{T\abs{x_n}\linf T\abs{x}}\\
& = & \norm{\abs{Tx_n}\linf \abs{w-Tx}}+\norm{T(\abs{x_n}\linf \abs{x})}\\
& \leq & \norm{w-Tx}+\norm{T}\norm{\abs{x_n}\linf \abs{x}}\\
& < & \frac{\epsilon}{2}+\norm{T}\frac{\epsilon}{2\norm{T}} = \epsilon ,
\end{eqnarray*}
therefore, $\{Tx_n\}$ is $un$-null so $T$ is  $u$-$M$-weakly compact.
%%%%%%%%%%%%
\item
The proof is clear by Proposition 3.5 of \cite{kandic2017un}.%\qedhere
\end{enumerate}
\end{proof}
%Proposition 2.7 of \cite{alavizadeh2018modu}
Recall that if $T:E\rightarrow F$ is an order bounded disjointness-preserving operator between two Banach lattices then $\abs{T}$ exists and is a lattice homomorphism. Therefore, by using \cite[Theorem 2.14]{aliprantis2006positive} and \cite[Theorem 3.1.4]{meyer2012banach} we have,
$\abs{\abs{T}x}=\abs{Tx}$ for all $x\in E$. In particular, for each $u\in F^+$ we have $\abs{\abs{T}x}\linf u=\abs{Tx}\linf u$.
% $\norm{\abs{T}x}=\norm{Tx}$ for all $x\in E$.
\begin{theorem}
\label{th:mod_disjoint}
Let $T:E\rightarrow F$ be an order bounded disjointness-preserving operator between two Banach lattices. Then the following assertions are hold.
\begin{enumerate}
\item $T$ is $u$-$M$-weakly compact if and only if $\abs{T}$ is $u$-$M$-weakly compact,\label{mod_mwc}
\item If $T$ or $\abs{T}$ is $u$-$L$-weakly compact then both of them are $u$-$L$- and $u$-$M$-weakly compact.\label{mod_lwc}
\end{enumerate}
\end{theorem}
\begin{proof}
\begin{enumerate}
\item%
Let $\{x_n\}$ be a norm bounded disjoint sequence in $E$. For each $u\in F^+$ and for each $n\in \mathbb{N}$, we have
$$\norm{\abs{\abs{T}x_n} \linf u}=\norm{\abs{Tx_n} \linf u}.$$
In other words, $\norm{Tx_n}\xrightarrow{un} 0$ if and only if $\norm{\abs{T}(x_n)}\xrightarrow{un} 0$. Thus, the proof is complete.
%%%%%%%%%%%
\item%
By using Proposition 2.7 of \cite{alavizadeh2018modu} $\sol(T(U))=\sol(\abs{T}(U))$, where $U$ is the closed unit ball of $E$. Therefore, $T$ is $u$-$L$-weakly compact if and only if $\abs{T}$ is $u$-$L$-weakly compact. Now, without loss of generality we assume that $\abs{T}$ is $u$-$L$-weakly compact. Let $\{x_n\}$ be a disjoint sequence in $U$. Since $\abs{T}$ is a lattice homomorphism, for $n\neq m$ we have
$$\abs{\abs{T}x_n}\linf \abs{\abs{T}x_m}=\abs{T}\abs{x_n}\linf \abs{T}\abs{x_m}=
\linf \abs{T}(\abs{x_n}\linf \abs{x_m})=0.$$
Therefore, $\{\abs{T}x_n\}$ is a disjoint sequence in $T(U)$. Since $\abs{T}$ is $u$-$L$-weakly compact, $\{\abs{T}x_n\}$ is $un$-convergent to zero. Thus, $\abs{T}$ is a $u$-$M$-weakly compact operator. It follows from previous part that $T$ is also $u$-$M$-weakly compact.%\qedhere
\end{enumerate}
\end{proof}
%%%%%%%%%%%%%%%%%%%%
In the following theorem we present a characterization of Banach lattices with order continuous norm.
\begin{theorem}\label{th:ocn}
Let $E$ be a Banach lattice. The following statements are equivalent.
\begin{enumerate}
\item $E$ has order continuous norm.\label{th:ocn:1}
\item $I:E\to E$ is $u$-$L$-weakly compact.\label{th:ocn:2}
\item $I:E\to E$ is $u$-$M$-weakly compact.\label{th:ocn:3}
\end{enumerate}
\end{theorem}
\section*{Conflict of interest}
The authors declare that they have no conflict of interest.


%\bibliographystyle{spmpsci}
%\bibliography{ourbibliography}

\end{document}